%% file: obstaclecikkarxiv.tex
\documentclass{article}
\usepackage{fullpage}   
\usepackage{authblk}    
\usepackage{amsmath}
\usepackage{amssymb}
\usepackage{graphicx}
\usepackage{amsthm}
\usepackage{graphicx, hyperref}
\usepackage{fixmath}
\usepackage{tikz}
\usepackage{subfigure}
\usepackage{mathptmx}      

\include{./mathdefs.harvard}

\newtheorem{theorem*}{Theorem}
\newtheorem{claim*}[theorem*]{Claim}
\newtheorem{corollary*}[theorem*]{Corollary}
\newtheorem{lemma*}[theorem*]{Lemma}

\title{Lower bounds on the obstacle number of graphs\footnote{The authors gratefully acknowledge support
from the Bernoulli Center at EPFL and from the Swiss National
Science Foundation, Grant No. 200021-125287/1. Research by the
second author was also supported by grants from NSF, BSF, and
OTKA.\newline
email: padmini.mvs@gmail.com, pach@cims.nyu.edu,
dom@cs.elte.hu}}

\author{Padmini Mukkamala}
\affil{Rutgers, The State University of New Jersey}
\author{J{\'{a}}nos Pach}
\affil{Ecole Polytechnique F\'ed\'erale de Lausanne}
\author{D\"om\"ot\"or P\'alv\"olgyi}
\affil{E\"otv\"os University Budapest}

\begin{document}

\maketitle

\begin{abstract}
Given a graph $G$, an {\em obstacle representation} of $G$ is a
set of points in the plane representing the vertices of $G$,
together with a set of connected obstacles such that two vertices
of $G$ are joined by an edge if and only if the corresponding
points can be connected by a segment which avoids all obstacles.
The {\em obstacle number} of $G$ is the minimum number of
obstacles in an obstacle representation of $G$. It is shown that
there are graphs on $n$ vertices with obstacle number at least
$\Omega({n}/{\log n})$.

\end{abstract}

\section{Introduction}

Consider a set $P$ of points in the plane and a set of closed
polygonal obstacles whose vertices together with the points in $P$
are in {\em general position}, that is, no {\em three} of them are
on a line. The corresponding {\em visibility graph} has $P$ as its
vertex set, two points $p,q\in P$ being connected by an edge if
and only if the segment $pq$ does not meet any of the obstacles.
Visibility graphs are extensively studied and used in
computational geometry, robot motion planning, computer vision,
sensor networks, etc.; see \cite{BKOS00}, \cite{G07}, \cite{OR97},
\cite{O99}, \cite{Ur00}.

Alpert, Koch, and Laison \cite{AKL09} introduced an interesting
new parameter of graphs, closely related to visibility graphs.
Given a graph $G$, we say that a set of points and a set of
polygonal obstacles as above constitute an {\em obstacle
representation} of $G$, if the corresponding visibility graph is
isomorphic to $G$. A representation with $h$ obstacles is also
called an $h$-obstacle representation. The smallest number of
obstacles in an obstacle representation of $G$ is called the {\em
obstacle number} of $G$ and is denoted by ${\rm obs}(G)$. Alpert
et al. \cite{AKL09} proved that there exist graphs with
arbitrarily large obstacle numbers.

Using tools from extremal graph theory, it was shown in \cite{PS}
that for any fixed $h$, the number of graphs with obstacle number
at most $h$ is $2^{o(n^2)}$. Notice that this immediately implies
the existence of graphs with arbitrarily large obstacle numbers.

In the present note, we establish some more precise estimates.

\begin{theorem*}\label{enumeration}
 (i) For any positive integer $h$, the number of graphs on $n$
(labeled) vertices that admit a representation with $h$ obstacles
is at most $$2^{O(hn log^2 n)}.$$
                  (ii) Moreover, the number of graphs on $n$ (labeled)
vertices that admit a representation with a set of obstacles
having a total of $s$ sides, is at most $$2^{O(n log n + s log
s)}.$$
\end{theorem*}

In the above bounds, it makes no difference whether we count
labeled or unlabeled graphs, because the number of labeled graphs
is at most $n!=2^{O(n log n)}$  times the number of unlabeled
ones.

It follows from Theorem~\ref{enumeration} (i) that for every $n$,
there exists a graph $G$ on $n$ vertices with obstacle number
$${\rm obs}(G)\geq \Omega\left({n}/{\log^2n}\right).$$
Indeed, as long as $2^{O(hn\log^2n)}$ is smaller than
$2^{\Omega(n^2)}$, the total number of (labeled) graphs with $n$
vertices, we can find at least one graph on $n$ vertices with
obstacle number $h$.

Here we show the following slightly stronger bound.

\begin{theorem*}\label{concave}
For every $n$, there exists a graph $G$ on $n$ vertices with
obstacle number
$${\rm obs}(G)\geq \Omega\left({n}/{\log n}\right).$$
\end{theorem*}

This comes close to answering the question in \cite{AKL09} whether
there exist graphs with $n$ vertices and obstacle number at least
$n$. However, we have no upper bound on the maximum obstacle
number of $n$-vertex graphs, better than $O(n^2)$.

Our next theorem answers another question from \cite{AKL09}.

\begin{theorem*}\label{exactly}
For every $h$, there exists a graph with obstacle number exactly
$h$.
\end{theorem*}

A special instance of the obstacle problem has received a lot of
attention, due to its connection to the Szemer\'edi-Trotter
theorem on incidences between points and lines~\cite{ST83a},
\cite{ST83b}, and other classical problems in incidence
geometry~\cite{PA95}. This is to decide whether the obstacle
number of ${\overline{K}}_n$, the empty graph on $n$ vertices, is
$O(n)$ if the obstacles must be {\em points}. The best known upper
bound is $n2^{O(\sqrt{\log n})}$ is due to Pach \cite{Pach03}; see
also Dumitrescu et al.~\cite{DPT09}, Matou\v sek~\cite{M09}, and
Aloupis et al. \cite{A+10conf}.

It is an interesting open problem to decide whether the obstacle
number of planar graphs can be bounded from above by a constant. For
outerplanar graphs, this has been verified by Fulek, Saeedi, and
Sar{\i}\"oz~\cite{FSS}, who proved that every outerplanar graph has
obstacle number at most $5$.

Theorem $i$ is proved in Section $i+1,\; 1\le i\le 3$.

\section{Proof of Theorem \ref{enumeration}}\label{sec:convex}
We will prove the theorem by a simple counting method.
Before turning to the proof, we
introduce some terminology. Given any placement (embedding) of the
vertices of $G$ in general position in the plane, a {\em straight-line drawing} or, in short, a {\em drawing}
of $G$ consists of the image of the embedding and the set of {\em
open line segments} connecting all pairs of points that correspond to the edges of $G$. If there is no danger of confusion, we make no
notational difference between the vertices of $G$ and the
corresponding points, and between the pairs $uv$ and the
corresponding open segments. The complement of the set of all
points that correspond to a vertex or belong to at least one edge
of $G$ falls into connected components. These components are
called the {\em faces} of the drawing. Notice that if $G$ has an
obstacle representation with a particular placement of its vertex
set, then

(1) each obstacle must lie entirely in one face of the drawing,
and

(2) each non-edge of $G$ must be blocked by at least one of the
obstacles.\\

We start by proving a result about the {\em convex obstacle number}
(a special case of Theorem \ref{concave}), as the arguments are
simpler here. Then we tackle Theorem \ref{enumeration} using similar methods.

Following Alpert et al., we define the {\em convex
obstacle number} ${\rm obs}_c(G)$  of a
graph $G$ as the minimal number of obstacles in an obstacle
representation of $G$, in which each obstacle is convex.

\begin{claim*}\label{convex}
For every $n$, there exists a graph $G$ on $n$ vertices with
convex obstacle number
$${\rm obs}_c(G)\geq \Omega\left({n}/{\log n}\right).$$
\end{claim*}

The idea is to find a short encoding of the obstacle
representations of graphs, and to use this to give an upper bound
on the number of graphs with low obstacle number. The proof uses
the concept of order types. Two sets of points, $P_1$ and $P_2$,
in general position in the plane are said to have the same {\em
order type} if there is a one to one correspondence between them
with the property that the orientation of any triple in $P_1$ is
the same as the orientation of the corresponding triple in $P_2$.
Counting the number of different order types is a classical task.

\medskip
\noindent {\bf Theorem A.} [Goodman, Pollack~\cite{GP86}] {\em
The number of different order types of $n$ points in general
position in the plane is $2^{O(n\log n)}$. }\medskip

\noindent Observe that asymptotically the same upper bound holds for the number
of different order types of $n$ {\em labeled} points, because the
number of different permutations of $n$ points is $n!=2^{O(n\log
n)}$.

\begin{proof}[Proof of Claim~\ref{convex}]
We will give an upper bound for the number of graphs that admit a
representation with at most $h$ convex obstacles. Let us fix such
a graph $G$, together with a representation. Let $V$ be the set of
points representing the vertices, and let $O_1,\ldots, O_h$ be the
convex obstacles. For any obstacle $O_i$, rotate an oriented
tangent line $\ell$ along its boundary in the clockwise direction.
We can assume without loss of generality that $\ell$ never passes
through two points of $V$. Let us record the sequence of points
met by $\ell$. If $v\in V$ is met at the right side of $\ell$, we
add the symbol $v^+$ to the sequence, otherwise we add $v^-$. (See
Figure 1.)

\input{convexObstacle.tex}

When $\ell$ returns to its initial position, we stop. The
resulting sequence consists of $2n$ characters. From this
sequence, it is easy to reconstruct which pairs of vertices are
visible in the presence of the single obstacle $O_i$. Hence,
knowing these sequences for every obstacle $O_i$, completely
determines the visibility graph $G$. The number of distinct
sequences assigned to a single obstacle is at most $(2n)!$, so the
number of graphs with convex obstacle number at most $h$ cannot
exceed $((2n)!)^h/h!<(2n)^{2hn}$. As long as this number is
smaller than $2^{n\choose 2}$, there is a graph with convex
obstacle number larger than $h$.
\end{proof}

To prove Theorem \ref{enumeration}, we will need one more result.
Given a drawing of a graph, the \emph{complexity} of a face is the number of line-segment sides bordering it. The following result was
proved by Arkin, Halperin, Kedem, Mitchell, and Naor (see Matou\v
sek, Valtr~\cite{MV97} for its sharpness).

\medskip
\noindent {\bf Theorem B.} [Arkin et al.~\cite{AHK95}] {\em
The complexity of a single face in a drawing of a graph with $n$
vertices is at most $O(n\log n)$. }
\medskip

\noindent Note that this bound does not depend on the number of
edges of the graph.

\begin{proof}[Proof of Theorem~\ref{enumeration}] First we show how to reduce part (i) of the theorem to part (ii). For each graph $G$ with $n$ vertices that admits a representation with at most $h$ obstacles, fix such a representation. Consider the visibility graph $G$ of the vertices in this representation. As explained at the beginning of this section, each obstacle belongs to a single face in this drawing. In view of Theorem~B, the complexity of every face is $O(n\log n)$. Replacing each obstacle by a slightly shrunken copy of the face containing it, we can achieve that every obstacle {\em is} a polygonal region with $O(n\log n)$ sides.

Now we prove part (ii). Notice that the order type of the sequence
$S$ starting with the vertices of $G$, followed by the vertices of
the obstacles (listed one by one, in cyclic order, and properly
separated from one another), completely determines $G$. That is, we
have a sequence of length $N$ with $N\le n + s$. According to
Theorem~A (and the comment following it), the number of different
order types with this many points is at most
$$2^{O(N\log N)}<2^{c(n+s)\log (n+s)},$$
for a suitable constant $c>0$. This is a very generous upper
bound: most of the above sequences do not correspond to any
visibility  graph $G$.
\end{proof}

Following Alpert et al., we define the {\em segment obstacle number} ${\rm obs}_s(G)$ of a graph $G$ as the minimal number of obstacles in an obstacle representation of $G$, in which each obstacle is a
{\em (straight-line) segment}. If we only allow segment obstacles, we have $s=2n$, and thus Theorem \ref{enumeration} (ii) implies the following bound.

\begin{corollary*}\label{segment}
For every $n$, there exists a graph $G$ on $n$ vertices with
segment obstacle number
$${\rm obs}_s(G)\geq \Omega\left({n^2}/{\log n}\right).$$
\end{corollary*}

In general, as long as the sum of the sides of the obstacles, $s$,
satisfies $s\log s=o({n\choose 2})$, we can argue that there is a
graph that cannot be represented with such obstacles.

\section{Proof of Theorem \ref{concave}}\label{sec:concave}

Before turning to the proof, we need a simple property of obstacle representations.

\begin{lemma*}\label{lemma} 
Let $k>0$ be an integer and let $G$ be a graph with $n$ vertices that has an obstacle representation with fewer than $\frac{n}{2k}$ obstacles. Then $G$ has at least $\lfloor\frac{n}{2k}\rfloor$ vertex disjoint induced subgraphs of $k$ vertices with obstacle number at most one.
\end{lemma*}

\begin{proof}
Fix an obstacle representation of $G$ with fewer than $\frac{n}{2k}$
obstacles. Suppose without loss of generality that in this representation no two vertices have the same $x$-coordinate. Using vertical lines, divide the vertices of $G$ into $\lfloor\frac{n}{k}\rfloor$ groups of size $k$ and possibly a last group that contains fewer than $k$ vertices. Let $G_1, G_2, \ldots$ denote the subgraphs of $G$ induced by these groups. Notice that if the convex hull of the vertices of $G_i$ does not entirely contain an obstacle, then ${\rm obs}(G_i)\le 1$. Therefore, the number of subgraphs $G_i$ that have $k$ points and obstacle number at most one is larger than $\lfloor\frac{n}{k}\rfloor - \frac{n}{2k}$, and the lemma is true.
\end{proof}

We prove Theorem \ref{concave} by a probabilistic argument.

\begin{proof}[Proof of Theorem \ref{concave}]
Let $G$ be a random graph on $n$ labeled vertices, whose edges are chosen independently with probability 1/2. Let $k$ be a positive integer to be specified later. According to Lemma~\ref{lemma},
$${\mbox{\rm Prob}}[{\rm obs}(G)<n/(2k)]$$ 
can be estimated from above by the probability that $G$ has at least
$\lfloor\frac{n}{2k}\rfloor$ vertex disjoint induced subgraphs of $k$ vertices such that each of them has obstacle number at most one. Let $p(n,k)$ denote the probability that $G$ satisfies this latter condition.

Suppose that $G$ has $\lfloor n/(2k)\rfloor$ vertex disjoint induced subgraphs $G_1, G_2,\ldots$ with $|V(G_i)|=k$ and ${\rm obs}(G_i)\le 1$. The vertices of $G_1, G_2,\ldots$ can be chosen in
at most ${n\choose k}^{\lfloor n/(2k)\rfloor}$ different ways. It follows from Theorem~\ref{enumeration}(i) that the probability that  a fixed $k$-tuple of vertices in $G$ induces a subgraph with obstacle number at most one is at most 
$$2^{O(k\log^2k)-{k\choose 2}}.$$
For disjoint $k$-tuples of vertices, the events that the obstacle numbers of their induced subgraphs do not exceed one are independent.

Therefore, we have
$$p(n,k)\le  {n\choose k}^{\lfloor n/(2k)\rfloor}\cdot
(2^{O(k\log^2k)-{k\choose 2}})^{\lfloor n/(2k)\rfloor}\le
2^{n\log n - nk/4 + O(n\log^2k)}.$$ 
Setting $k=\lfloor 5\log n\rfloor)$, the right-hand side of the last inequality tends to zero. In this case, almost all graphs on $n$ vertices have obstacle number at least $\frac{n}{2k}>\frac{n}{10\log n}$, which completes the proof.
\end{proof}

\section{Proof of Theorem \ref{exactly}}\label{sec:exactly}
Alpert, Koch, and Laison \cite{AKL09} asked whether for every natural number $h$ there exists a graph whose obstacle number is exactly $h$. Here we answer this question in the affirmative.

\begin{proof}[Proof of Theorem \ref{exactly}]
Pick a graph $G$ with obstacle number $h'>h$. (The existence of
such a graph was first proved in \cite{AKL09}, but it also follows 
from Theorem \ref{concave}.) Let $n$ denote the number of
vertices of $G$. Consider the complete graph $K_n$ on $V(G)$. Clearly, ${\rm obs}(K_n)=0$, and $G$ can be obtained from $K_n$
by successively deleting edges. Observe that as we delete an edge
from a graph $G'$, its obstacle number cannot increase by more
than {\em one}. Indeed, if we block the deleted edge $e$ by adding a very small obstacle that does not intersect any other edge of $G'$, we obtain a valid obstacle representation of $G'-e$. (Of course, the obstacle number of a graph can also {\em decrease} by the removal of an edge.) At the beginning of the process, $K_n$ has obstacle number {\em zero}, at the end $G$ has obstacle number $h'>h$, and whenever it increases, the increase is {\em one}. We can conclude that at some stage we obtain a graph with obstacle number precisely $h$.
\end{proof}

The same argument applies to the convex obstacle number, to the
segment obstacle number, and many similar parameters.

\subsection*{Acknowledgement.}

We are indebted to Den{\.{i}}z Sar{\i}{\"{o}}z for many valuable
discussions on the subject. A conference version containing some
results from this paper and some from \cite{PS} appeared in
\cite{MPS}.

\bibliographystyle{plain}
\bibliography{obs}
\end{document}

%% file: mathdefs.harvard.tex
\def\be{\begin{equation}}
\def\ee{\end{equation}}
\def\bea{\begin{eqnarray}}
\def\eea{\end{eqnarray}}
\def\bean{\begin{mathletters}\begin{eqnarray}}
\def\eean{\end{eqnarray}\end{mathletters}}

\newcommand{\tbox}[1]{\mbox{\tiny #1}}

\def\text{\tbox}

\def\Es2{E_{0,{\rm sat}}^2}

\newcounter{eqletter}
\def\mathletters{%
\setcounter{eqletter}{0}%
\addtocounter{equation}{1}
\edef\curreqno{\arabic{equation}}
\edef\@currentlabel{\theequation}
\def\theequation{%
\addtocounter{eqletter}{1}\thechapter.\curreqno\alph{eqletter}%
}%
}

%% file: convexObstacle.tex
\begin{figure*}[ht]
\label{fig:convex}
{\centering
\subfigure[Empty]{
\begin{tikzpicture}[scale=0.35]
\filldraw [blue!80!black!20!white] (0,0) -- (1,1) -- (2.5,1) -- (3.5,0) -- (2.5,-1) -- (1,-1) -- (0,0) -- cycle;
\node [fill=black,circle,inner sep=1pt,label=left:$1$] (1) at (-1,0) {}; 
\node [fill=black,circle,inner sep=1pt,label=left:$2$] (2) at (2,3) {}; 
\node [fill=black,circle,inner sep=1pt,label=left:$3$] (3) at (3,-2.5) {}; 
\draw [<-,red] (0,4) -- (0,-3);
\draw [->,green] (-0.5,4) arc (145:35:22pt);
\end{tikzpicture}
}
\qquad
\subfigure[$2^+$]{
\begin{tikzpicture}[scale=0.35]
\filldraw [blue!80!black!20!white] (0,0) -- (1,1) -- (2.5,1) -- (3.5,0) -- (2.5,-1) -- (1,-1) -- (0,0) -- cycle;
\node [fill=black,circle,inner sep=1pt,label=left:$1$] (1) at (-1,0) {}; 
\node [fill=red,circle,inner sep=1pt,label=left:$2$] (2) at (2,3) {}; 
\node [fill=black,circle,inner sep=1pt,label=left:$3$] (3) at (3,-2.5) {}; 
\draw [->,red] (-2,-3) -- (2.5,3.75);
\end{tikzpicture}
}
\qquad
\subfigure[$2^+1^-$]{
\begin{tikzpicture}[scale=0.35]
\filldraw [blue!80!black!20!white] (0,0) -- (1,1) -- (2.5,1) -- (3.5,0) -- (2.5,-1) -- (1,-1) -- (0,0) -- cycle;
\draw [->,red] (-1.75,-0.375) -- (3.75,2.375);
\node [fill=red,circle,inner sep=1pt,label=left:$1$] (1) at (-1,0) {}; 
\node [fill=black,circle,inner sep=1pt,label=left:$2$] (2) at (2,3) {}; 
\node [fill=black,circle,inner sep=1pt,label=left:$3$] (3) at (3,-2.5) {}; 
\end{tikzpicture}
}
\qquad
\subfigure[$2^+1^-2^-$]{
\begin{tikzpicture}[scale=0.35]
\filldraw [blue!80!black!20!white] (0,0) -- (1,1) -- (2.5,1) -- (3.5,0) -- (2.5,-1) -- (1,-1) -- (0,0) -- cycle;
\node [fill=black,circle,inner sep=1pt,label=left:$1$] (1) at (-1,0) {}; 
\node [fill=red,circle,inner sep=1pt,label=left:$2$] (2) at (2,3) {}; 
\node [fill=black,circle,inner sep=1pt,label=left:$3$] (3) at (3,-2.5) {}; 
\draw [->,red] (1.5,4) -- (5,-3);
\end{tikzpicture}
}
\qquad
\subfigure[$2^+1^-2^-3^+$]{
\begin{tikzpicture}[scale=0.35]
\filldraw [blue!80!black!20!white] (0,0) -- (1,1) -- (2.5,1) -- (3.5,0) -- (2.5,-1) -- (1,-1) -- (0,0) -- cycle;
\node [fill=black,circle,inner sep=1pt,label=left:$1$] (1) at (-1,0) {}; 
\node [fill=black,circle,inner sep=1pt,label=left:$2$] (2) at (2,3) {}; 
\node [fill=red,circle,inner sep=1pt,label=left:$3$] (3) at (3,-2.5) {}; 
\draw [<-,red] (2.75,-3.75) -- (4.25,3.75);
\end{tikzpicture}
}
\qquad
\subfigure[$2^+1^-2^-3^+1^+$]{
\begin{tikzpicture}[scale=0.35]
\filldraw [blue!80!black!20!white] (0,0) -- (1,1) -- (2.5,1) -- (3.5,0) -- (2.5,-1) -- (1,-1) -- (0,0) -- cycle;
\node [fill=red,circle,inner sep=1pt,label=left:$1$] (1) at (-1,0) {}; 
\node [fill=black,circle,inner sep=1pt,label=left:$2$] (2) at (2,3) {}; 
\node [fill=black,circle,inner sep=1pt,label=left:$3$] (3) at (3,-2.5) {}; 
\draw [<-,red] (-2,0.5) -- (4.5,-2.75);
\end{tikzpicture}
}
\qquad
\subfigure[$2^+1^-2^-3^+1^+3^-$]{
\begin{tikzpicture}[scale=0.35]
\filldraw [blue!80!black!20!white] (0,0) -- (1,1) -- (2.5,1) -- (3.5,0) -- (2.5,-1) -- (1,-1) -- (0,0) -- cycle;
\node [fill=black,circle,inner sep=1pt,label=left:$1$] (1) at (-1,0) {}; 
\node [fill=black,circle,inner sep=1pt,label=left:$2$] (2) at (2,3) {}; 
\node [fill=red,circle,inner sep=1pt,label=left:$3$] (3) at (3,-2.5) {}; 
\draw [<-,red] (-1.5,0.875) -- (5,-4);
\end{tikzpicture}
}
\qquad
\subfigure[$2^+1^-2^-3^+1^+3^-$]{
\begin{tikzpicture}[scale=0.35]
\filldraw [blue!80!black!20!white] (0,0) -- (1,1) -- (2.5,1) -- (3.5,0) -- (2.5,-1) -- (1,-1) -- (0,0) -- cycle;
\node [fill=black,circle,inner sep=1pt,label=left:$1$] (1) at (-1,0) {}; 
\node [fill=black,circle,inner sep=1pt,label=left:$2$] (2) at (2,3) {}; 
\node [fill=black,circle,inner sep=1pt,label=left:$3$] (3) at (3,-2.5) {}; 
\draw (3,-2.5) -- (-1,0) -- (2,3);
\draw [dashed,red](3,-2.5) -- (2,3);
\node (4) at (5,0) {};
\end{tikzpicture}
}
\caption{Parts (a) to (g) show the construction of the sequence and (h) shows the visibilities. The arrow on the tangent line indicates the direction from the point of tangency in which we assign $+$ as a label to the vertex. The additional arrow in (a) indicates that the tangent line is rotated clockwise around the obstacle.}
} 
\end{figure*}
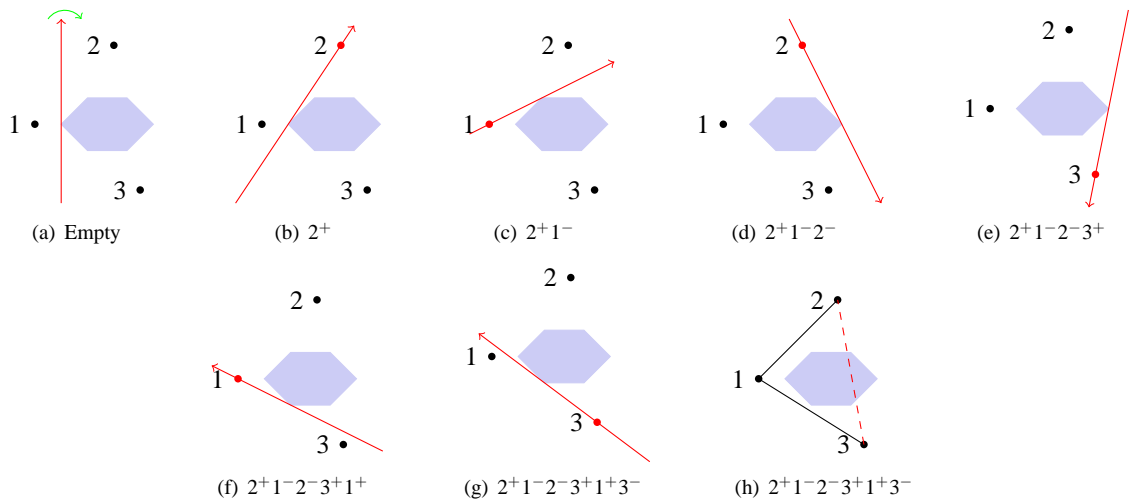